\documentclass[11pt]{article}
\usepackage{epigamath}

\usepackage[notext]{kpfonts}
\usepackage{baskervald}

\setpapertype{A4}

\usepackage[english]{babel}

\usepackage[dvips]{graphicx}     

\title{\vspace{-0.5cm}Stable rationality of higher dimensional conic bundles}
\titlemark{Stable rationality of conic bundles}
\author{\vspace{0cm} Hamid Ahmadinezhad and Takuzo Okada}
\authoraddresses{
\authordata{Hamid Ahmadinezhad}{\firstname{Hamid} \lastname{Ahmadinezhad}\\
\institution{Department of Mathematical Sciences, Loughborough University, LE11 3TU, United Kingdom}\\
\email{h.ahmadinezhad@lboro.ac.uk}}\\
\authordata{Takuzo Okada}{\firstname{Takuzo}
\lastname{Okada}\\
\institution{Department of Mathematics, Faculty of Science and Engineering, Saga University, Saga 840-8502 Japan}\\
\email{okada@cc.saga-u.ac.jp}}
}
\authormark{H. Ahmadinezhad and T. Okada}
\date{\vspace{-5ex}} 
\journal{\'Epijournal de G\'eom\'etrie Alg\'ebrique} 
\acceptation{Received by the Editors on February 7, 2018, and in final form
on April 25, 2018. \\ Accepted on May 21, 2018.}

\newcommand{\articlereference}{Volume 2 (2018), Article Nr. 5}

 \usepackage[all]{xy}

\allowdisplaybreaks

\newtheorem{Thm}{Theorem}[section]
\newtheorem{Lem}[Thm]{Lemma}
\newtheorem{Cor}[Thm]{Corollary}
\newtheorem{Prop}[Thm]{Proposition}

\newtheorem{Conj}[Thm]{Conjecture}

\newtheorem{Def}[Thm]{Definition}
\newtheorem{Def-Lem}[Thm]{Definition-Lemma}

\newtheorem{Rem}[Thm]{Remark}

\newcommand{\Spec}{\operatorname{Spec}}

\newcommand{\Cox}{\operatorname{Cox}}
\newcommand{\Cl}{\operatorname{Cl}}

\newcommand{\mbA}{\mathbb{A}}

\newcommand{\mbC}{\mathbb{C}}

\newcommand{\mbP}{\mathbb{P}}
\newcommand{\mbQ}{\mathbb{Q}}

\newcommand{\mbZ}{\mathbb{Z}}

\newcommand{\mcB}{\mathcal{B}}

\newcommand{\mcE}{\mathcal{E}}

\newcommand{\mcL}{\mathcal{L}}
\newcommand{\mcM}{\mathcal{M}}
\newcommand{\mcN}{\mathcal{N}}
\newcommand{\mcO}{\mathcal{O}}

\newcommand{\mcQ}{\mathcal{Q}}

\newcommand{\mcX}{\mathcal{X}}

\newcommand{\msp}{\mathsf{p}}
\newcommand{\msq}{\mathsf{q}}
\newcommand{\K}{\Bbbk}

\newcommand{\inj}{\hookrightarrow}

\newcommand{\ratmap}{\dashrightarrow}

\newcommand{\CH}{\operatorname{CH}}

\begin{document}


\maketitle



\begin{prelims}


\def\abstractname{Abstract}
\abstract{We prove that a very general nonsingular conic bundle $X \to \mathbb{P}^{n-1}$ embedded in a projective vector bundle of rank $3$ over $\mathbb{P}^{n-1}$ is not stably rational if the anti-canonical divisor of $X$ is not ample and $n \ge 3$.}

\keywords{Stable Rationality; conic Bundles}

\MSCclass{14E08}

\vspace{0.15cm}

\languagesection{Fran\c{c}ais}{%

\textbf{Titre. Rationalit\'e stable des fibr\'es en coniques de grande dimension} \commentskip \textbf{R\'esum\'e.} Nous d\'emontrons qu'un fibr\'e en coniques non-singulier tr\`es g\'en\'eral
$X\to\mathbb{P}^{n-1}$ plong\'e dans le projectivis\'e d'un fibr\'e vectoriel de rang $3$ au dessus de $\mathbb{P}^{n-1}$ n'est pas stablement rationnel lorsque le diviseur anti-canonique de $X$ n'est pas ample et $n\geq 3$.}

\end{prelims}


\newpage

\setcounter{tocdepth}{1} \tableofcontents

\section{Introduction} \label{sec:intro}

An important question in algebraic geometry is to determine whether an algebraic variety is rational; that is, birational to projective space. Two algebraic varieties are said to be birational if they become isomorphic after removing finitely many lower-dimensional subvarieties from both sides. The closest varieties to being rational are those that admit a fibration into a projective space with all fibres rational curves; so-called conic bundles. 

In this article, we study stable (non-)rationality of conic bundles  over a projective space of arbitrary dimension (greater than one). A non-rational variety $X$ may become rational after being multiplied by a suitable projective space, i.e., $X\times \mbP^m$ is birational to $\mbP^{n+m}$, where $n=\dim X$, in which case we say $X$ is stably rational.

Stable non-rationality of conic bundles in dimension $3$ has been studied extensively in \cite{ABBP, BB} and \cite{HKT}, giving a satisfactory answer.
In higher dimensions almost nothing is known except for a few examples of stably non-rational conic bundles over $\mbP^3$ given in \cite{ABBP} and \cite{HPT}.

Throughout this article, by a conic bundle we mean a Mori fibre space of relative dimension $1$ (see Definition \ref{def:cb} for details). The following is our main result.

\begin{Thm} \label{thm:ample}
Let $n \ge 3$ and $d$ be integers, and let $\mcE$ be a direct sum of three invertible sheaves on $\mbP^{n-1}$.
Let $X$ be a very general member of a complete linear system $|2 D + d F|$ on $\mbP_{\mbP^{n-1}} (\mcE)$, where $D$ is the tautological divisor and $F$ is the pullback of the hyperplane on $\mbP^{n-1}$.
Suppose that the natural projection $X \to \mbP^{n-1}$ is a conic bundle.
\begin{itemize}
\item[\rm (1)] If $X$ is singular, then $X$ is rational.
\item[\rm (2)] If $X$ is non-singular and $-K_X$ is not ample, then $X$ is not stably rational.
\end{itemize}
\end{Thm}

This result covers the following varieties as a special case.

\begin{Cor} \label{cor1}
Let $X$ be a very general hypersurface of bi-degree $(d,2)$ in $\mbP^{n-1} \times \mbP^2$.
If $d \ge n \ge 3$, then $X$ is not stably rational.
\end{Cor}

This can be thought of as a higher dimensional generalisation of the main result of \cite{BB}.

\begin{Cor} \label{cor2}
Let $X$ be a double cover of $\mbP^{n-1} \times \mbP^1$ branched along a very general divisor of bi-degree $(2d,2)$.
If $2 d \ge n \ge 3$, then $X$ is not stably rational.
\end{Cor}

%
%

By a result of Sarkisov \cite{Sarkisov}, a conic bundle is birational to a standard conic bundle which is by definition a nonsingular conic bundle flat over a smooth base. The following criterion for rationality in terms of the discriminant was conjectured by Shokurov \cite{Shokurov} (see also \cite[Conjecture~I]{Iskovskikh}). Remarkabe progress toward this conjecture has been made in \cite{Iskovskikh} and \cite{MP}.

\begin{Conj}[{\cite[Conjecture~10.3]{Shokurov}}]\label{Isk}
Let $X \to S$ be a $3$-dimensional standard conic bundle and $\Delta \subset S$ the discriminant divisor.
If $|2 K_S + \Delta| \ne \emptyset$, then $X$ is not rational. 
\end{Conj}

Although the statement becomes weaker than Theorem \ref{thm:ample}, we can restate our main result in terms of the discriminant:

\begin{Cor} \label{cor:discr}
With notation and assumptions as in \emph{Theorem \ref{thm:ample}}, assume in addition that $X$ is nonsingular and let $\Delta \subset \mbP^{n-1}$ be the discriminant divisor of the conic bundle $X \to \mbP^{n-1}$.
\begin{itemize}
\item[\rm (1)] If $|3 K_{\mbP^{n-1}} + \Delta| \ne \emptyset$, then $X$ is not stably rational.
\item[\rm (2)] If $n \ge 7$, $\pi \colon X \to \mbP^{n-1}$ is standard and $|2 K_{\mbP^{n-1}} + \Delta| \ne \emptyset$, then $X$ is not stably rational.
\end{itemize}
\end{Cor} 

This leads us to pose the following.

\begin{Conj}
Let $\pi \colon X \to S$ be an $n$-dimensional standard conic bundle with $n \ge 3$.
If $\left| 2 K_S + \Delta \right| \ne \emptyset$, then $X$ is not rational.
If in addition $X$ is very general in its moduli, then $X$ is not stably rational.
\end{Conj}

\paragraph*{\bf The argument of stable non-rationality.}
It is known that a stably rational smooth projective variety is universally $\CH_0$-trivial; see \cite[Lemme~1.5]{CTP1} and \cite[theorem~1.1]{Totaro} and references therein.
Let $\mathcal{X}\rightarrow\mathcal{B}$ be a flat family over a complex curve $\mathcal{B}$ with smooth general fibre. Then, by the specialisation theorem of Voisin \cite[Theorem~2.1]{Voisin}, the stable non-rationality of a very general fibre will follow if the special fibre $X_0$ is not universally $\CH_0$-trivial and has at worst ordinary double point singularities.  This was generalised by Colliot-Th\'el\`ene and Pirutka \cite[Th\'eor\`eme 1.14]{CTP1} to the case where  
\begin{itemize}
\item[\rm 1.] $X_0$ admits a universally $\CH_0$-trivial resolution $\varphi \colon Y \to X_0$ such that $Y$ is not universally $\CH_0$-trivial,
\item[\rm 2.] in mixed characteristic, that is, when $\mcB = \Spec A$ with $A$ being a  DVR of possibly mixed characteristic.
\end{itemize}
Thus it is enough to verify the existence of such a resolution $\varphi \colon Y \to X_0$ over an algebraically closed field of characteristic $p > 0$.
In view of \cite[Lemma~2.2]{Totaro}, the core of the proof of universal $\CH_0$-nontriviality for $Y$ in our case is done by showing that $H^0(Y,\Omega^i)\neq 0$ for some $i>0$, following Koll\'ar \cite{Kollar1} and Totaro \cite{Totaro}. This is done in Section~\ref{stable-rationality}.

\paragraph*{Acknowledgements.}
We would like to thank Professor Jean-Louis Colliot-Th\'el\`ene  and Professor Vyacheslav Shokurov for pointing out two oversights in an earlier version of this article. We would also like to thank the anonymous referee for helpful comments. The second author is partially supported by JSPS KAKENHI Grant Number 26800019.

\section{Embedded conic bundles}

\subsection{Weighted projective space bundles}

In this subsection we work over a field $k$.

\begin{Def} \label{def:wpsbundle}{\rm
A {\it toric weighted projective space bundle over} $\mbP^n$ is a projective simplicial toric variety with Cox ring
\[
\Cox (P) = k [u_0,\dots,u_n,x_0,\dots,x_m],
\]
which is $\mbZ^2$-graded as
\[
\begin{pmatrix}
1 & \cdots & 1 & \lambda_0 & \cdots & \lambda_m \\
0 & \cdots & 0 & a_0 & \cdots & a_m
\end{pmatrix}
\]
with the irrelevant ideal $I = (u_0,\dots,u_n) \cap (x_0,\dots,x_m)$, where $\lambda_0,\dots,\lambda_m$ are integers and $n$, $m$, $a_0,\dots,a_m$ are positive integers.
In other words, $P$ is the geometric quotient 
\[
P = (\mbA^{n+m+2} \setminus V (I))/\mathbb{G}_m^2,
\]
where the action of $\mathbb{G}_m^2 = \mathbb{G}_m \times \mathbb{G}_m$ on $\mbA^{n+m+2} = \Spec \Cox (P)$ is given by the above matrix.}
\end{Def}

The natural projection $\Pi\colon P \to \mbP^{n}$ by the coordinates $u_0,\dots,u_n$ realizes $P$ as a $\mbP (a_0,\dots,a_m)$-bundle over $\mbP^n$.
In this paper, we simply call $P$ the {\it $\mbP (a_0,\dots,a_m)$-bundle over $\mbP^n$ defined by}
\[
\begin{pmatrix}
u_0 & \cdots & u_n & & x_0 & \cdots & x_m \\
1 & \cdots & 1 & | & \lambda_0 & \cdots & \lambda_m \\
0 & \cdots & 0 & | & a_0 & \cdots & a_m
\end{pmatrix}.
\]

In the following, let $P$ be as in Definition \ref{def:wpsbundle}.
Let $\msp \in P$ be a point and $\msq \in \mbA^{n+m+2} \setminus V (I)$ a preimage of $\msp$ via the morphism $\mbA^{n+m+2} \setminus V (I) \to P$.
We can write $\msq = (\alpha_0,\dots,\alpha_n,\beta_0,\dots,\beta_m)$, where $\alpha_i,\beta_j \in k$.
In this case we express $\msp$ as $(\alpha_0\!:\!\cdots\!:\!\alpha_n;\beta_0\!:\!\cdots\!:\!\beta_m)$.

\begin{Rem}{\rm
We will frequently use the following coordinate change.
Consider a point $\msp = (\alpha_0\!:\!\cdots\!:\!\alpha_n;\beta_0\!:\!\cdots\!:\!\beta_m) \in P$ and suppose for example that $\alpha_0 \ne 0, \beta_j \ne 0$ and $a_j = 1$ for some $j$.
Then for $l \ne j$ such that $\lambda_l/a_l \ge \lambda_j$, the replacement
\[
x_l \mapsto \alpha_0^{\lambda_l - a_l \lambda_j} \beta_j^{a_l} x_l - \beta_l u_0^{\lambda_l - a_l \lambda_j} x_j^{a_l}
\]
induces an automorphism of $P$.
By considering the above coordinate change, we can transform $\msp$ (via an automorphism of $P$) into a point for which the $x_l$-coordinate is zero for $l$ with $\lambda_l/a_l \ge \lambda_j$.}
\end{Rem}

We have the decomposition
\[
\Cox (P) = \bigoplus_{(\alpha,\beta) \in \mbZ^2} \Cox (P)_{(\alpha,\beta)},
\]
where $\Cox (P)_{(\alpha,\beta)}$ consists of the homogeneous elements of bi-degree $(\alpha,\beta)$.
An element $f \in \Cox (P)_{(\alpha,\beta)}$ is called a (homogeneous) polynomial of bi-degree $(\alpha,\beta)$.

The Weil divisor class group $\Cl (P)$ is naturally isomorphic to $\mbZ^2$.
Let $F$ and $D$ be the divisors on $P$ corresponding to $(1,0)$ and $(0,1)$, respectively, which generate $\Cl (P)$.
Note that $F$ is the class of the pullback of a hyperplane on $\mbP^{n}$ via $\Pi\colon P \to \mbP^{n}$.
We denote by $\mcO_P (\alpha,\beta)$ the rank $1$ reflexive sheaf corresponding to the divisor class of type $(\alpha,\beta)$, that is, the divisor $\alpha F + \beta D$.
More generally, for a subscheme $Z \subset P$, we set $\mcO_Z (\alpha,\beta) = \mcO_P (\alpha,\beta)|_Z$.
We remark that there is an isomorphism
\[
H^0 (P,\mcO_P (\alpha,\beta)) \cong \Cox (P)_{(\alpha,\beta)}.
\]


\begin{Def}{\rm
For integers $k,l,m,n$ with $n \ge 3$, we define $P_n (k,l,m)$ (resp.\ $Q_n (k,l)$) to be the $\mbP^2$-bundle (resp.\ $\mbP^1$-bundle) over $\mbP^{n-1}$ defined by the matrix
\[
\begin{pmatrix}
u_0 & \cdots & u_{n-1} & & x & y & z \\
1 & \cdots & 1 & | & k & l & m \\
0 & \cdots & 0 & | & 1 & 1 & 1
\end{pmatrix}
\quad
\left(\text{resp.}
\begin{pmatrix}
u_0 & \cdots & u_{n-1} & & x & y \\
1 & \cdots & 1 & | & k & l \\
0 & \cdots & 0 & | & 1 & 1
\end{pmatrix}
\right).
\]}
\end{Def}

\begin{Rem} \label{rem:projvb} {\rm
Let $P$ be as in Definition \ref{def:wpsbundle}.
When $a_0 = \cdots = a_m = 1$, $P$ is a $\mbP^m$-bundle over $\mbP^n$.
More precisely we have an isomorphism
\[
P \cong \mbP_{\mbP^n} (\mcO_{\mbP^n} (-\lambda_0) \oplus \cdots \oplus \mcO_{\mbP^n} (-\lambda_m)).
\]
Here, for a vector bundle $\mcE$ over $\mbP^n$, $\mbP (\mcE) =\mbP_{\mbP^n} (\mcE)$ denotes the projective bundle of one-dimensional quotients of $\mcE$.
Moreover, via the above isomorphism, the pullback of a hyperplane on $\mbP^{n-1}$ and the tautological divisor on $\mbP (\mcE)$ are identified with the divisors on $P$ corresponding to $(1,0)$ and $(0,1)$, respectively.}
\end{Rem}

\subsection{Embedded conic bundles}


In the rest of this section we work over $\mbC$.
By a {\it splitting vector bundle}, we mean a vector bundle which is a direct sum of invertible sheaves.

\begin{Def} \label{def:cb} {\rm
Let $X$ be a normal projective $\mbQ$-factorial variety of dimension $n$.
We say that a morphism $\pi \colon X \to \mbP^{n-1}$ is a {\it conic bundle} (over $\mbP^{n-1}$) if it is a Mori fibre space, that is, $X$ has only terminal singularities, $\pi$ has connected fibres, $-K_X$ is $\pi$-ample and $\rho (X) = 2$, where $\rho (X)$ denotes the rank of the Picard group.

An {\it embedded conic bundle} $\pi \colon X \to \mbP^{n-1}$ is a conic bundle such that $X$ is embedded in a projective bundle $\mbP (\mcE) :=\mbP_{\mbP^{n-1}} (\mcE)$ as a member of $|d F + 2 D|$ for some splitting vector bundle $\mcE$ of rank $3$ on $\mbP^{n-1}$ and $d \in \mbZ$, and $\pi$ coincides with the restriction of $\Pi\colon\mbP (\mcE) \to \mbP^{n-1}$ to $X$.
Here $F$ and $D$ denote the pullback of a hyperplane on $\mbP^{n-1}$ and the tautological class $D$ on $\mbP (\mcE)$, respectively.}
\end{Def}

In the following let $\mcE$ be a splitting vector bundle of rank $3$ on $\mbP^{n-1}$ and $X \in |d F + 2 D|$ be a general member.
We denote by $\pi \colon X \to \mbP^{n-1}$ the restriction of $\Pi\colon\mbP (\mcE) \to \mbP^{n-1}$ to $X$.
Without loss of generality we may assume that 
\[
\mcE \cong \mcO_{\mbP^{n-1}} (-k) \oplus \mcO_{\mbP^{n-1}} (-l) \oplus \mcO_{\mbP^{n-1}} (-m)
\] 
for some $k \le l \le m$.
Then, by Remark \ref{rem:projvb}, we have $\mbP (\mcE) \cong P_n (k,l,m)$ and the linear system $|d F + 2 D|$ on $\mbP_{\mbP^{n-1}} (\mcE)$ corresponds to $|\mcO_{P_n (k,l,m)} (d,2)|$.
Here we do not assume that $\pi \colon X \to \mbP^{n-1}$ is a conic bundle. 
We study conditions on $k,l,m$ and $d$ that make $\pi \colon X \to \mbP^{n-1}$ a conic bundle.

\begin{Lem} \label{lem:crisingcbfam}
Let $k, l, m, d$ be integers such that $k \le l \le m$.
Set $P = P_n (k,l,m)$ and let $X$ be a general member of $|\mcO_P (d,2)|$.
\begin{itemize}
\item[\rm (1)] $X$ is smooth if and only if $d \ge 2 m$, $d = l + m$, or $d = k + m$.
\item[\rm (2)] $X$ is not smooth and has only terminal singularities if and only if $2 m > d > l+m$.
\item[\rm (3)] $X$ is non-normal if and only if $k + m > d$. 
\end{itemize}
\end{Lem}

\begin{proof}
Suppose that $d \ge 2 m$.
Then $|\mcO_P (d,2)|$ is base point free and its general member $X$ is smooth.
In the following we assume that $2 m > d \ge k + m$.

Suppose that $2 m > d > l + m$.
Then $X$ is defined in $P$ by
\[
a x^2 + b y^2 + f x y + g x z + h y z = 0,
\]
where $a,b,f,g,h \in \mbC [u]$.
We have $\deg h = d - (l+m) > 0$ and $\deg g = d - (k+m) > 0$.
Then $X$ is singular along $(x = y = g = h = 0) \ne \emptyset$.
The singular locus is of codimension $3$ in $X$.
Since $X$ is general, the hypersurfaces in $\mbP^{n-1}$ defined by $g = 0$ and $h = 0$ are both nonsingular and intersect transversally.
It is then straightforward to check that the blowup $\sigma \colon X' \to X$ along the singular locus is a resolution and we have $K_{X'} = \sigma^*K_X + E$, where $E$ is the exceptional divisor.
Thus $X$ has terminal singularities.

Suppose that $2 m > d = l + m$.
Then $X$ is defined in $P$ by
\[
a x^2 + b y^2 + f x y + g x z + y z = 0.
\]
Replacing $y$ and $z$ suitably, we can eliminate the terms $b y^2, f x y$ and $g z x$, that is, $X$ is defined by 
\[
a x^2 + y z = 0.
\]
It is then clear that $X$ is smooth, when $a$ is general.

Suppose that $l + m > d > k + m$.
Then $X$ is defined in $P$ by
\[
a x^2 + b y^2 + f x y + g x z = 0.
\]
We have $\deg g = d - (k + m) > 0$.
Then $X$ is singular along $(x = y = g = 0) \ne \emptyset$, and the singularity is not terminal since the singular locus is of codimension $2$ in $X$.

Suppose that $l + m > d = k+m$.
Then $X$ is defined in $P$ by
\[
a x^2 + b y^2 + f x y + x z = 0.
\]
Replacing $z$ suitably, we may assume that $X$ is defined by
\[
b y^2 + z x = 0.
\]
It is easy to see that $X$ is smooth.

Finally suppose that $k + m > d$.
Then $X$ is defined in $P$ by
\[
a x^2 + b y^2 + f x y = 0,
\]
where $a,b,f \in \mbC [u]$.
In this case $X$ is singular along the divisor $(x = y = 0) \subset X$.
Thus $X$ is not normal.
The above arguments prove (1), (2) and (3).
\hfill $\Box$ \end{proof}

\begin{Lem} \label{lem:embcbpic}
In the same setting as in \emph{Lemma \ref{lem:crisingcbfam}}, suppose that either $d = l + m$ or $d = k + m$.
Then the variety $X$ is rational.
Moreover we have $\rho (X) \ge 3$ unless $k = l = m$.
\end{Lem}

\begin{proof}
Suppose that $d = l + m$, which implies $2 m \ge d = l + m$.
We claim that $X$ is defined by an equation of the form $a x^2 + y z = 0$, where $a \in \mbC [u]$.
This is already proved in Lemma \ref{lem:crisingcbfam}, when $2 m > d$.
Suppose that $2 m = d = l + m$.
Then $l = m$ and $X$ is defined by
\[
a x^2 + y^2 + z^2 + f x y + g x z + \alpha y z = 0,
\]   
where $\alpha \in \mbC$ and $a,f,g \in \mbC [u]$.
Replacing $y$ and $z$, the above equation can be transformed into $a x^2 + y z = 0$ and the claim is proved.

We consider the projection $X \ratmap Q := Q_n (k,l)$
Note that $Q \cong \mbP (\mcO (-k) \oplus \mcO (-l))$.
Then the projection is birational, hence $X$ is rational.
The projection $X \ratmap Q$ is defined outside $(x = y = 0) \subset X$.
Let $\msp \in (x = y = 0)$ be a point.
Then $z$ does not vanish at $\msp$ and we have
\[
y = \frac{y z}{z} = - \frac{a x^2}{z}.
\]
From this we deduce that $X \ratmap Q$ is everywhere defined.
Now we assume that either $k \ne l$ or $l \ne m$.
Then $\deg a = d - 2 k = l + m - k > 0$.
We see that $(y = a = 0) \subset X$ is a divisor and it is contracted by $X \to Q$ to a codimension $2$ subset of $Q$.
This shows $\rho (X) \ge 3$.

Next, suppose that $d = k + m$.
Note that $l + m \ge d$.
If in addition $l + m > d$, then, by the proof of Lemma \ref{lem:crisingcbfam}, the defining equation of $X$ can be written as $b y^2 + x z = 0$.
The statement follows from the same argument as above.
If $l + m = d$, then $k = l$ and we have $d = l + m$.
This case is already proved.
\hfill $\Box$ \end{proof}

\begin{Lem} \label{lem:classifnonsingcb}
In the same setting as in \emph{Lemma \ref{lem:crisingcbfam}}, $\pi \colon X \to \mbP^{n-1}$ is a nonsingular conic bundle if and only if one of the following holds:
\begin{itemize}
\item[\rm (1)] $d > 2 m$, 
\item[\rm (2)] $d = 2 m$ and $m > l$, or 
\item[\rm (3)] $d = 2 m = 2 l = 2 k$.
\end{itemize}
\end{Lem}

\begin{proof}
This follows from Lemmas \ref{lem:crisingcbfam} and \ref{lem:embcbpic}.
\hfill $\Box$ \end{proof}

\begin{Prop} \label{prop:ratsingfam}
Let $X$ be an embedded conic bundle over $\mbP^{n-1}$.
If $X$ is general $($in the linear system$)$ and singular, then $X$ is rational.
\end{Prop}

\begin{proof}
We may assume that $X \in |\mcO_P (d,2)|$, where $P = P_n (k,l,m)$, for some $k \le l \le m$.
By Lemma \ref{lem:crisingcbfam}, we have $2 m > d \ge k + m$.
Then a general member $X$ is defined by an equation of the form
\[
a x^2 + b y^2  + f x y + g x z + h y z = 0,
\]
where $a,b,f,g,h \in \mbC [u]$.
Here, note that, if for example $l + m > d$, then we know that the term $h y z$ does not appear in the equation.
The inequality $d \ge k +m$ implies that $g \ne 0$ since $X$ is general.
Let $P \ratmap Q = Q_n (k,l)$ be the natural projection.
Now we can write the defining equation as
\[
z (g x + h y) + a x^2 + b y^2 + f x y = 0,
\]
which implies that the restriction $X \ratmap Q$ is birational.
Therefore $X$ is rational.
\hfill $\Box$ \end{proof}

The following can be considered as a ``normal form'' of conic bundles, which describes nonsingular embedded conic bundles (see Proposition \ref{prop:nonsingcbfam}).

\begin{Def}{\rm
Let $(\lambda,\mu,\nu)$ be a triplet of integers $\lambda,\mu,\nu$.
We say that $\pi \colon X \to \mbP^{n-1}$ (or $X$) is of type $[\lambda,\mu,\nu]$ if $X$ belongs to $|\mcO_P (\lambda+\mu+\nu,2)|$, where $P = P_n (\lambda,\mu,\nu)$, and $\pi$ coincides with the restriction of $P \to \mbP^{n-1}$ to $X$.}
\end{Def}

\begin{Prop} \label{prop:nonsingcbfam}
Let $\pi \colon X \to \mbP^{n-1}$ be a nonsingular embedded conic bundle.
Then $X$ is either of type $[\lambda,\mu,\nu]$ for some $\lambda,\mu,\nu$ such that $0 < \lambda \le \mu \le \nu \le \lambda + \mu$ or of type $[0,0,0]$.
\end{Prop}

\begin{proof}
We may assume that $X$ belongs to $|\mcO_{P_n (k,l,m)} (d,2)|$ for some $k \le l \le m$ and $d$.
Since the family $X$ is non-singular, we have $d \ge 2 m$ by Lemma \ref{lem:classifnonsingcb} and $X$ is defined in $P_n (k,l,m)$ by an equation of the form 
\[
a x^2 + b y^2 + c z^2 + f x y + g x z + h y z = 0,
\]
where $a,b,c,f,g,h \in \mbC [u]$. 
We set $\alpha = \deg a, \beta = \deg b, \gamma = \deg c, \lambda = \deg h, \mu = \deg g$ and $\nu = \deg f$.
By comparing the weights, we have 
\[
\alpha + 2 k = \beta + 2 l = \gamma + 2 m = \nu + k + l = \mu + k + m = \lambda + l + m.
\]
Now we have
\[
P_n (k,l,m) \cong P_n (k + (\nu - m), l + (\nu - m), m + (\nu - m)) \cong P_n (\lambda,\mu,\nu) =: P,
\]
and the linear system $|\mcO_{P_n (k,l,m)} (d,2)|$ is identified with $|\mcO_P (\lambda + \mu + \nu,2)|$.
Thus $X$ is of type $[\lambda,\mu,\nu]$.
By applying Lemma \ref{lem:classifnonsingcb} for $k = \lambda, l = \mu, m = \nu$ and $d = \lambda + \mu + \nu$, we get the desired result.
\hfill $\Box$ \end{proof}

\begin{Rem}{\rm
In the language of \cite[Definition 3.1]{ABBP}, a conic bundle $\pi \colon X \to \mbP^{n-1}$ of type $[\lambda,\mu,\nu]$ with $\lambda \le \mu \le \nu \le \lambda + \mu$ is a conic bundle of graded-free type over $\mbP^{n-1}$ corresponding to the triplet $(-\lambda + \mu + \nu,\lambda - \mu + \nu, \lambda + \mu - \nu)$. }
\end{Rem}

\section{Stable non-rationality}\label{stable-rationality}

In this section we study stable (non-)rationality of nonsingular embedded conic bundles $\pi \colon X \to \mbP^{n-1}$.
By Proposition \ref{prop:nonsingcbfam}, such a conic bundle is of type $[\lambda,\mu,\nu]$, where either $0 < \lambda \le \mu \le \nu \le \lambda + \mu$ or $\lambda = \mu = \nu = 0$.
In case $X$ is of type $[0,0,0]$, then $X \cong \mbP^{n-1} \times \mbP^1$ and it is obviously rational.
We consider the remaining cases and thus we assume that 
\[
0 < \lambda \le \mu \le \nu \le \lambda + \mu 
\]
throughout this section.
In addition we assume $\nu \ge 3$ throughout.

We set $P = P_n (\lambda,\mu,\nu)$, $\delta = \lambda + \mu + \nu$, and consider special members $X \in |\mcO_P (\delta,2)|$ defined in $P$ by an equation of the form
\begin{equation} \label{defeq}
a x^2 + b y^2 + c z^2 + f x y = 0,
\end{equation}
where $a,b,c,f$ are general polynomials in variables $u_0,\dots,u_{n-1}$.
Recall that $\nu = \deg f$ and $\deg a = - \lambda + \mu + \nu$, $\deg b = \lambda - \mu + \nu$ and $\deg c = \lambda + \mu - \nu$.

\begin{Rem} \label{rem:num} {\rm
By the assumptions on $\lambda,\mu,\nu$, we have $\deg a = - \lambda + \mu + \nu \ge 3$, $\deg b = \lambda - \mu + \nu \ge 1$, $\deg c = \lambda + \mu - \nu \ge 0$ and $\deg f = \nu \ge 3$.}
\end{Rem}

\begin{Lem} \label{lem:nonsingchar0}
If the ground field is an algebraically closed field of characteristic $0$, then $X$ is smooth.
\end{Lem}

\begin{proof}
The variety $X$ is a general member of the base point free sub linear system of $|\mcO_P (\delta,2)|$ on the smooth variety $P$.
Thus, by the Bertini theorem, a general $X$ is smooth.
\hfill $\Box$ \end{proof}

We use universal $\CH_0$-triviality to test stable rationality of varieties.

\begin{Def} {\rm
Let $V$ be a projective variety defined over a field $k$.
We denote by $\CH_0 (V)$ the {\it Chow group of $0$-cycles} on $V$.
We say that $V$ is {\it universally $\CH_0$-trivial} if for any field $F$ containing $k$, the degree map $\CH_0 (V_F) \to \mbZ$ is an isomorphism.
A projective morphism $\varphi \colon W \to V$ defined over $k$ is {\it universally $\CH_0$-trivial} if for any field containing $k$, the push-forward map $\varphi_* \colon \CH_0 (W_F) \to \CH_0 (V_F)$ is an isomorphism.}
\end{Def}

In the rest of this section we work over an algebraically closed field $\K$ of characteristic $2$.
Let $R$ be the $\mbP (1,1,2)$-bundle over $\mbP^{n-1}$ defined by
\[
\begin{pmatrix}
u_0 & u_1 & \cdots & u_{n-1} & & x & y & \bar{z} \\
1 & 1 & \cdots & 1 & | & \lambda & \mu & 2 \nu \\
0 & 0 & \cdots & 0 & | & 1 & 1 & 2
\end{pmatrix}
\]
and let $Z \subset R$ be the hypersurface defined by
\[
a x^2 + b y^2 + c \bar{z} + f x y = 0.
\]
We have a natural morphism $P \to R$ which is a (purely inseparable) double cover branched along $(\bar{z} = 0) \subset R$.
The image of $X$ under $P \to R$ is the hypersurface $Z \subset R$. Let $\tau \colon X \to Z$ be the induced morphism, which is a double cover branched along the divisor cut out on $Z$ by $\bar{z} = 0$.
We set $\mcL = \mcO_Z (\nu,1)$.
Then $\bar{z}$ is a global section of $\mcL^2$, and over the non-singular locus of $Z$, $\tau$ is the double cover obtained by taking the roots of $\bar{z} \in H^0 (Z,\mcL^2)$ in the sense of \cite[Construction 8]{Kollar1}.

In Sections \ref{sec:sing} and \ref{sec:crit} below we will analyse the singularities of $X$ and $Z$, and finally we will show the existence of a universally $\CH_0$-trivial resolution $\varphi \colon Y \to X$ such that $H^0 (Y,\Omega_Y^{n-1}) \ne 0$ under some conditions on $\lambda,\mu,\nu$.
The latter implies that $Y$ is not universally $\CH_0$-trivial by \cite[Lemma 2.2]{Totaro}.

\subsection{Singularities} \label{sec:sing}

Recall that the ground field $\K$ is an algebraically closed field of characteristic $2$ and $X$ is a hypersurface in $P = P_n (\lambda,\mu,\nu)$ defined by
\[
a x^2 + b y^2 + c z^2 + f x y = 0
\]
for general $a,b,c,f \in \K [u_0,\dots,u_{n-1}]$.
Similarly $Z$ is the hypersurface in $R$ defined by
\[
a x^2 + b y^2 + c \bar{z} + f x y = 0.
\]
We set 
\[
\Xi = (x = y = 0) \subset R, \ 
\Xi_Z = \Xi \cap Z = (x = y = c = 0),
\]
and $R^{\circ} = R \setminus \Xi$, $Z^{\circ} = Z \setminus \Xi_Z$.

In order to analyze singularities of $Z^{\circ} \subset R^{\circ}$, we consider standard affine charts of $R^{\circ}$.
For $i = 0,\dots,n-1$ and a coordinate $w \in \{x,y\}$, we set $U_{u_i,w} = (u_i \ne 0) \cap (w \ne 0) \subset R^{\circ}$.
We have
\[
R^{\circ} = \bigcup_{i \in \{0,\dots,n-1\}, w \in \{x,y\}} U_{u_i,w}.
\]
We remark that $U_{u_i,w}$ is an affine $(n+1)$-space and that the restriction of the sections 
\[
\{u_0,\dots,u_{n-1}, x,y,\bar{z}\} \setminus \{u_i,w\}
\] 
are affine coordinates of $U_{u_i,w}$.
We only treat $U_{u_0,x}$ because the other open subsets can be understood by symmetry.
We set
\[
\tilde{u}_i = u_i/u_0, \ 
\tilde{y} = y/x u_0^{\mu - \lambda}, \ 
\tilde{z} = \bar{z}/x^2 u_0^{\nu - 2 \lambda}.
\]
Then $U_{u_0,w}$ is an affine $(n+1)$-space with affine coordinates $\tilde{u}_1,\dots,\tilde{u}_{n-1},\tilde{y},\tilde{z}$.
By a slight abuse of notation, the affine coordinates $\tilde{u}_1,\dots,\tilde{u}_{n-1},\tilde{y},\tilde{z}$ are simply denoted by $u_1,\dots,u_{n-1},y,\bar{z}$.

\begin{Lem} \label{lem:smZ}
$Z^{\circ}$ is smooth.
\end{Lem}

\begin{proof}
If $\deg c = 0$, then $c$ is a non-zero constant and thus $\Xi_Z = \emptyset$.
In this case $Z = Z^{\circ}$ is a $\mbP^1$ bundle over $\mbP^{n-1}$ and it is smooth.

In the following we assume that $\deg c > 0$ and set 
\[
U_x = (x \ne 0), \ 
U_y = (y \ne 0) \subset R,
\]
so that $R^{\circ} = U_x \cup U_y$.
We will show that for any point $\msq \in R^{\circ}$, the condition that $Z^{\circ}$ is singular at $\msq \in Z$ imposes $n+2$ independent conditions on $a,b,c,f$. 
Then the assertion will follow by a dimension count argument since $\dim R^{\circ} = n+1$.
We note that $\deg b = \lambda - \mu + \nu \ge 1$, $\deg c = \lambda + \mu - \nu \ge \lambda \ge 3$ and $\deg f = \lambda \ge 3$ by Remark \ref{rem:num}.

Let $\msq \in U_x$.
Replacing coordinates, we may assume $\msq = (1\!:\!0\!:\!\cdots\!:\!0 ; 1\!:\!0\!:\!0)$.
Then $U_{u_0,x} \subset Q$ is an affine space with coordinates $u_1,\dots,u_{n-1},y,\bar{z}$ and $Z \cap U_{u_0,z}$ is defined by
\[
\tilde{a} + \tilde{b} y^2 + \tilde{c} \bar{z} + \tilde{f} y = 0,
\]
where we set $\tilde{h} = h (1,u_1,\dots,u_{n-1})$ for a polynomial $h (u_0,\dots,u_{n-1})$.
Note that $\msq$ corresponds to the origin.
The variety $Z^{\circ}$ is singular at $\msq$ if and only if $\tilde{a}, \tilde{c}, \tilde{f}$ vanish at $\msq$ and the linear part of $\tilde{a}$ is zero. 
This imposes $n+2$ independent conditions since $\deg a > 0$ and $\deg c, \deg f \ge 0$ (cf.\ Remark \ref{rem:num}).

Suppose that $\msq \in U_y$.
Replacing coordinates, we may assume $\msq = (1\!:\!0\!:\!\cdots\!:\!0 ; 0\!:\!1\!:\!0)$.
Then $U_{u_0,y} \subset Q$ is an affine space with coordinates $u_0,\dots,u_{n-1},x,\bar{z}$ and $Z \cap U_{u_0,y}$ is defined by
\[
\tilde{a} x^2 + \tilde{b} + \tilde{c} \bar{z} + \tilde{f} x = 0.
\]
The variety $Z^{\circ}$ is singular at $\msq$ if and only if $\tilde{b}, \tilde{c}, \tilde{f}$ vanish at $\msq$ and the linear part of $\tilde{b}$ is zero. 
The latter imposes $n+2$ independent conditions since $\deg b > 0$ and $\deg c, \deg f \ge 0$ (cf.\ Remark \ref{rem:num}), and the proof is complete.
\hfill $\Box$ \end{proof}

We set $X^{\circ} = \pi^{-1} (Z^{\circ})$.

\begin{Lem} \label{lem:smX}
$X$ is smooth along $X \setminus X^{\circ}$.
\end{Lem}

\begin{proof}
Note that $X \setminus X^{\circ} = X \cap (x = y = 0)$.
For a point $\msp \in X \setminus X^{\circ}$, $X$ is smooth at $\msp$ if and only if the hypersurface $(c = 0) \subset \mbP^{n-1}$ is smooth at the image of $\msp$ under $X \to \mbP^{n-1}$. 
Clearly the hypersurface $(c = 0) \subset \mbP^{n-1}$ is smooth since $c$ is general, and the assertion follows.
\hfill $\Box$ \end{proof}

\subsection{Analysis of critical points} \label{sec:crit}

We set $\mcL^{\circ} = \mcL|_{Z^{\circ}}$, where we recall $\mcL = \mcO_Z (\nu,1)$.
By Lemma \ref{lem:smZ}, $Z^{\circ}$ is non-singular and by Koll\'ar's result \cite[V.5]{Kollar2} there exists an invertible sheaf $\mcQ^{\circ}$ on $Z^{\circ}$ such that $\mcM^{\circ} := \tau^*\mcQ^{\circ} \subset (\Omega_{X^{\circ}}^{n-1})^{\vee \vee}$, where $\vee \vee$ denotes the double dual.
Let $\mcM$ be the push-forward of the invertible sheaf $\mcM^{\circ}$ via the open immersion $X^{\circ} \inj X$.
By Lemma \ref{lem:smX}, $\mcM$ is an invertible sheaf on $X$.

\begin{Def}{\rm
Let $V$ be a nonsingular variety of dimension $n$ defined over an algebraically closed field $\K$ of characteristic $2$, $\mcN$ an invertible sheaf on $V$ and $s \in H^0 (V, \mcN^2)$ a section.
Let $\msp \in V$ be a point, $\xi$ a local generator of $\mcN$ at $\msp$ and $s = f (x_1,\dots,x_n) \xi^2$ a local description of $s$ with respect to local coordinates $x_1,\dots,x_n$ of $V$ at $\msp$.
We say that $s$ has a {\it critical point} at $\msp$ if the linear term of $f$ is zero.

We say that $s$ has an {\it admissible critical point} at $\msp$ if for a suitable choice of coordinates $x_1,\dots,x_n$, 
\[
f =
\begin{cases} 
\alpha + x_1 x_2 + x_3 x_4 + \cdots + x_{n-1} x_n + g, & \text{if $n$ is even}, \\
\alpha + \beta x_1^2 + x_2 x_3 + \cdots + x_{n-1} x_n + g, & \text{if $n$ is odd},
\end{cases}
\]
where $\alpha, \beta \in \K$, $g = g (x_1,\dots,x_n) \in (x_1,\dots,x_n)^3$ and, in case $n$ is odd, the coefficient of $x_1^3$ in $g$ is nonzero.}
\end{Def}

\begin{Lem} \label{lem:admcrit} 
The section $\bar{z} \in H^0 (Z, \mcL^2)$ has only admissible critical points on $Z^{\circ}$.
\end{Lem}

\begin{proof}
We choose and fix a general $c \in \K [u]$ so that the hypersurface $(c = 0) \subset \mbP^{n-1}$ is non-singular.
Clearly $\bar{z}$ does not have a critical point on $(c = 0) \subset Z^{\circ}$.
On $Z^{\circ} \cap (c \ne 0)$, the section $c$ is invertible and thus the section $\bar{z}$ has an admissible critical point if and only if the section
\[
s := c (a x^2 + b y^2 + f x y) \ (= c^2 \bar{z})
\]
has an admissible critical point.
It is then enough to show that the section $s$, viewed as a section on $Q = Q_n (\lambda,\mu)$, has only admissible critical points on $U_c = (c \ne 0) \subset Q$ for general $a,b$ and $f$.
We set $U_x = (x \ne 0) \subset Q$ and $\Pi_y = (x = 0) \cap (y \ne 0) \subset Q$ so that $Q = U_x \cup \Pi_y$.

We first show that $s$ does not have a critical point on $\Pi_y \cap U_c$.
Let $\msp \in \Pi_y \cap U_c$ be a point.
We may assume $\msp = (1\!:\!0\!:\!\cdots\!:\!0 ; 0\!:\!1)$.
We work on the open subset $U_{u_0,y} = (u_0 \ne 0) \cap (y \ne 0) \subset Q$ which is the affine space with coordinates $u_1,\dots,u_{n-1}$ and $x$.
For $e = e (u_0,\dots,u_{n-1})$, we set $\tilde{e} = e (1,u_1,\dots,u_{n-1})$.
Moreover we denote by $\tilde{e}_i$ the degree $i$ part of $\tilde{e}$.
Then the restriction of $s$ to $U_{u_0,y}$ is $\tilde{c} (\tilde{a} x^2 + \tilde{b} + \tilde{f} x)$ and the point $\msp$ corresponds to the origin.
Then $s$ has a critical point at $\msp$ if and only if 
\[
\tilde{c}_0 (\tilde{b}_1 + \tilde{f}_0 x) + \tilde{c}_1 \tilde{b}_0 = 0.
\]
Note that $\tilde{c}_0 \ne 0$.
Since $\deg b \ge 1$, this imposes $n$ independent conditions on $a,b,f$. 
Thus, for any point $\msp \in \Pi_y$, $n$ conditions are imposed in order for $s$ to have a critical point at $\msp$.
By counting dimensions we conclude that $s$ does not have a critical point on $\Pi_y \cap U_c$ since $\dim \Pi_y = n-1$.

Let $\msp \in U_x \cap U_c$ be a point.
We may assume $\msp = (1\!:\!0\!:\!\cdots\!:\!0 ; 1\!:\!0)$.
We work on the open subset $U_{u_0,x} = (u_0 \ne 0) \cap (x \ne 0) \subset R$ which is the affine space with coordinates $u_1,\dots,u_{n-1}$ and $y$.
We have $s|_{U_{u_0,y}} = \tilde{c} (\tilde{a} + \tilde{b} y^2 + \tilde{f} y)$.
Let $\ell, q$ and $h$ be the linear, quadratic and cubic parts of $s|_{U_{u_0,y}}$, respectively.
We have
\[
\ell = \tilde{c}_0 (\tilde{a}_1 + \tilde{f}_0 y) + \tilde{c}_1 \tilde{a}_0.
\]
Since $\deg a \ge 1$, $n$ conditions are imposed in order for $s$ to have a critical point at $\msp$.
It remains to show the existence of a section $s = c (a x^2 + b y^2 + f x y)$ which has an admissible critical point at $\msp$.
Now suppose that $s$ has a critical point at $\msp$, that is, $\ell = 0$.
This implies that $\tilde{f}_0 = 0$ and $\tilde{a}_1 = \tilde{a}_0 \tilde{c}_1/\tilde{c}_0$.
Then, for the quadratic and cubic parts, we have
\[
\begin{split}
q &= \tilde{c}_0 (\tilde{a}_2 + \tilde{b}_0 y^2 + \tilde{f}_1 y) + \frac{\tilde{a}_0 \tilde{c}_1^2}{\tilde{c}_0} + \tilde{c}_2 \tilde{a}_0, \\
h &= \tilde{c}_0 (\tilde{a}_3 + \tilde{b}_1 y^2 + \tilde{f}_2 y) + \cdots.
\end{split}
\]
Since $\deg a \ge 3$ and $\deg f \ge 3$, we can choose $a,b,f$ so that
\[
q = 
\begin{cases}
y u_1 + u_2 u_3 + u_4 u_5 +\cdots + u_{n-2} u_{n-1}, & \text{if $n$ is even}, \\
y u_1 + u_2 u_3 + u_4 u_5 + \cdots + u_{n-3} u_{n-2} + u_{n-1}^2, & \text{if $n$ is odd}.
\end{cases}
\]
In case $n$ is even, the section $s$ has a nondegenerate critical point at $\msp$ and we are done.
Suppose that $n$ is odd.
Since $\deg a \ge 3$, then we can choose $a,b,f$ so that $q$ is as above and the coefficient of $u_{n-1}^3$ in $h$ is non-zero.
For this choice of $a,b,c,f$, the section $s$ has an admissible critical point at $\msp$ and the proof is completed by the dimension counting argument.
\hfill $\Box$ \end{proof}

\begin{Prop} \label{prop:existresol}
Let the notation and assumption as above.
Assume in addition that $\nu \ge n$.
Then there exists a universally $\CH_0$-trivial resolution $\varphi \colon Y \to X$ of singularities such that $H^0 (Y,\Omega_Y^{n-1}) \ne 0$.
In particular $Y$ is not universally $\CH_0$-trivial.
\end{Prop}

\begin{proof}
By \cite[Proposition 4.1]{Okada} or \cite{CTP2}, if the singularities of $X$ correspond to admissible critical points of the section $\bar{z}$, then there exists a universally $\CH_0$-trivial resolution $\varphi \colon Y \to X$ such that $\varphi^*\mcM \inj \Omega_Y^{n-1}$ (in fact, $\varphi$ is just the composite of blowups at each (isolated) singular point).
Thus, by Lemma \ref{lem:admcrit}, $X$ admits such a resolution.
The branch divisor $(\bar{z} = 0)$ is clearly reduced and, by \cite[Lemma V.5.9]{Kollar2}, we have an isomorphism
\[
\mcM^{\circ} \cong \tau^* (\omega_{Z^{\circ}} \otimes {\mcL^{\circ}}^2)
\cong \mcO_{X^{\circ}} (\nu - n,0),
\]  
so that $\mcM \cong \mcO_X (\nu - n,0)$.
By the assumption we have $\nu - n \ge 0$, which implies $H^0 (X,\mcM) \ne 0$.
Thus $H^0 (Y,\Omega_Y^{n-1}) \ne 0$ and by \cite[Lemma 2.2]{Totaro}, $Y$ is not universally $\CH_0$-trivial.
\hfill $\Box$ \end{proof}


\subsection{Proof of theorems and corollaries}

\begin{Thm} \label{thm:genthm}
Suppose that the ground field is $\mbC$ and let $(\lambda,\mu,\nu)$ be a triplet of integers such that $0 < \lambda \le \mu \le \nu \le \lambda + \mu$.
If in addition $\nu \ge n$, then a very general embedded conic bundle $\pi \colon X \to \mbP^{n-1}$ of type $[\lambda,\mu,\nu]$ is not stably rational.
\end{Thm}

\begin{proof}
For a field (or more generally a ring) $K$, we denote by $P_K$ the $\mbP^2$-bundle $P_n (\lambda,\mu,\nu)$ over $\mbP^{n-1}$ defined over $K$.
Let $\K$ be an algebraically closed field of characteristic $2$ and let $X \to \mbP^{n-1}$ be a very general hypersurface in $P_{\K}$ defined by an equation of the form \eqref{defeq}.
We take a mixed characteristic discrete valuation ring $A$ whose residue field is $\K$, for example the Witt ring, and then we lift $X$ to a hypersurface $\mcX$ of $P_A$ defined  by an equation of the form \eqref{defeq}.
We choose and fix an embedding of the quotient field of $A$ into $\mbC$ and set $V = \mcX \times_A \mbC$.
Then $V$ is a very general hypersurface of $P_{\mbC}$ defined by an equation of the form \eqref{defeq}.
By Proposition \ref{prop:existresol}, we can apply the specialization theorem \cite[Th\'eor\`eme 1.14]{CTP1} and conclude that $V$ is not universally $\CH_0$-trivial.
Note that $V$ is nonsingular by Lemma \ref{lem:nonsingchar0}.
Note also that $V$ is not a very general conic bundle of type $[\lambda,\mu,\nu]$.
However a very general conic bundle of type $[\lambda,\mu,\nu]$ degenerates (over a complex curve) to $V$, hence the assertion follows from the specialization argument \cite[Theorem 2.1]{Voisin} (or by \cite[Theorem 4.2.10]{NS}). 
\hfill $\Box$ \end{proof}

Now we can prove the main theorem and corollaries in Section \ref{sec:intro}.

\begin{proof}[Proof of \emph{Theorem \ref{thm:ample}}]
The assertion (1) follows from Proposition \ref{prop:ratsingfam}.

Let $\pi \colon X \to \mbP^{n-1}$ be a non-singular embedded conic bundles over $\mbP^{n-1}$.
By Proposition \ref{prop:nonsingcbfam}, we may assume that it is of type $[\lambda,\mu,\nu]$, where either $0 < \lambda \le \mu \le \nu \le \lambda + \mu$ or $\lambda = \mu = \nu = 0$.
By adjunction we have $\mcO_X (-K_X) \cong \mcO_X (n,1)$. 
The complete linear system $|\mcO_P (n,1)|$, where $P = P_n (\lambda,\mu,\nu)$, is base point free if and only if $n \ge \nu$.
This shows that $\mcO_P (n,1)$, and hence $\mcO_X (n,1)$, is ample if $n < \nu$.
Since $-K_X$ is not ample by assumption, we have $n \ge \nu$.
Therefore (2) follows from Theorem \ref{thm:genthm}.
\hfill $\Box$ \end{proof}

\begin{proof}[Proof of \emph{Corollaries \ref{cor1} and \ref{cor2}}]
Let $X$ be a very general hypersurface of bi-degree $(d,2)$ in $\mbP^{n-1} \times \mbP^2$.
Then $\mcO_X (-K_X) \cong \mcO_X (n-d,1)$.
By assumption $d \ge n$ and this implies that $-K_X$ is not ample.
Thus $X$ is not stably rational by Theorem \ref{thm:ample}.

Let $X$ be a double cover of $\mbP^{n-1} \times \mbP^1$ branched along a very general divisor of bi-degree $(2 d,2)$.
Then $X$ is a very general member of $|\mcO_P (2 d,2)|$, where $P = P_n (0,0,d)$, and hence it is of type $[d,d,2d]$.
By the assumption we have $2 d \ge n$.
Thus $X$ is not stably rational by Theorem \ref{thm:genthm}
\hfill $\Box$ \end{proof}

\begin{proof}[Proof of \emph{Corollary \ref{cor:discr}}]
Let $\pi \colon X \to \mbP^{n-1}$ be as in Corollary \ref{cor:discr}.
Then we may assume that it is of type $[\lambda,\mu,\nu]$, where $0 < \lambda \le \mu \le \nu \le \lambda+\nu$ or $\lambda = \mu = \nu = 0$.
The discriminant divisor $\Delta$ is a hypersurface in $\mbP^{n-1}$ of degree $\lambda + \mu + \nu$.
The condition $|3 K_{\mbP^{n-1}} + \Delta| \ne \emptyset$ is equivalent to the condition $\lambda + \mu + \nu \ge 3n$.
The latter implies $\nu \ge n$ since $\lambda \le \mu \le \nu$.
Thus (1) follows from Theorem \ref{thm:genthm}.

Now suppose in addition that $n \ge 7$ and $\pi \colon X \to \mbP^{n-1}$ is standard.
Note that $X$ is defined in $P_n (\lambda,\mu,\nu)$ by an equation of the form 
\[
a x^2 + b y^2 + c z^2 + f x y + g x z + h y z = 0,
\]
where $a,\dots,h \in \mbC [u]$.
If $\deg c = \lambda + \mu - \nu > 0$, then the system of equations $a = b = \cdots = h = 0$ has a non-trivial solution on $\mbP^{n-1}$ since $n \ge 7$.
This implies that $\pi$ cannot be flat, in particular, not standard.
Thus $\nu = \lambda + \mu$ and in this case the condition $|2 K_{\mbP^{n-1}} + \Delta| = |\mcO_{\mbP^{n-1}} (2 (\nu - n))| \ne \emptyset$ is equivalent to $\nu \ge n$ which implies stable non-rationality of $X$ again by Theorem \ref{thm:genthm}.
This proves (2).
\hfill $\Box$ \end{proof}

\providecommand{\bysame}{\leavevmode\hbox to3em{\hrulefill}\thinspace}
%
%

\bibliographystyle{amsalpha}
\bibliographymark{References}
\def\cprime{$'$}

\end{document}